\documentclass[11pt]{amsart}
\usepackage{a4wide}
\usepackage[all,cmtip]{xy}
\usepackage{amsmath,amssymb, amsfonts}
\usepackage{mathabx}

\newtheorem{theorem}{Theorem}[section]
\newtheorem{theoremA}{Theorem}

\newtheorem{lemma}[theorem]{Lemma}
\newtheorem{proposition}[theorem]{Proposition}
\newtheorem{corollary}[theorem]{Corollary}

\theoremstyle{definition}
\newtheorem{definition}[theorem]{Definition}
\theoremstyle{remark}
\newtheorem{remark}[theorem]{Remark}
\newtheorem{example}[theorem]{Example}

\makeatletter
\@addtoreset{theoremA}{section}
\makeatother

\newcommand{\Lg}{\mathfrak{g}}

\newcommand{\Lie}{\mathfrak{Lie}}

\newcommand{\Lt}{\mathfrak{t}}
\newcommand{\K}{\mathop\mathrm{Ker}}
\newcommand{\I}{\mathop\mathrm{Im}}

\newcommand{\ham}{\mathfrak{ham}}

\newcommand{\norm}[1]{\Vert #1\Vert}

\newcommand{\Aut}{\mathrm{Aut}}

\newcommand{\aut}{\mathfrak{aut}}
\newcommand{\id}{\operatorname{\text{\sf id}}}

\newcommand{\T}{\mathcal{T}}

\newcommand{\LL}{\mathcal{L}}
\newcommand{\TT}{\mathbb{T}}

\newcommand{\al}{\alpha}

\newcommand{\la}{\lambda}

\renewcommand{\phi}{\varphi}

\newcommand{\ce}{\mathcal{C}^\infty}
\newcommand{\CC}{\mathbb{C}}

\newcommand{\RR}{\mathbb{R}}

\newcommand{\ZZ}{\mathbb{Z}}

\newcommand{\QQ}{\mathbb{Q}}

\newcommand{\Ka}{K\"{a}hler}

\begin{document}

\title{A characterisation of toric LCK manifolds}
\author{Nicolina Istrati}
\address{Univ Paris Diderot, Sorbonne Paris Cit\'{e}, Institut de Math\'{e}matiques de Jussieu-Paris Rive Gauche, UMR 7586, CNRS, Sorbonne Universit\'{e}s, UPMC Univ Paris 06, F-75013, Paris, France}
\email{nicolina.istrati@imj-prg.fr}
%\thanks{}
\date{\today}
\maketitle
\abstract 
We prove that a compact toric locally conformally \Ka\ manifold which is not \Ka\ admits a toric Vaisman structure, a fact which was conjectured in \cite{mmp}. This is the final step leading to the classification of compact toric locally conformally \Ka\ manifolds started in \cite{p} and \cite{mmp}. We also show, by constructing an example, that unlike in the symplectic case, toric locally conformally symplectic manifolds are not necessarily toric locally conformally \Ka.
\endabstract

\section{Introduction}

In the present paper, we are interested in the incarnation of toric geometry for locally conformally \Ka\ (LCK), or more generally, for locally conformally symplectic (LCS) manifolds. The beginnings of this study can be traced down to the article of I. Vaisman \cite{va1}, where he argues that LCS manifolds are the natural phase spaces for Hamiltonian mechanics and is the first to give a good notion of Hamiltonians in this context. 

An LCS structure on a smooth manifold is a non-degenerate two-form which, around every point of the manifold, differs from a local symplectic form by a conformal factor. In analogy, an LCK structure is a LCS structure together with a compatible complex structure, such that the non-degenerate form in this case is locally conformal to a \Ka\ form. As they appear in the current definitions and in the motivations given by I. Vaisman, LCS/LCK manifolds generalise symplectic/\Ka\ manifolds. However, in many ways, and in particular in the context of this paper, there exists a duality between the behaviour of strict LCS manifolds and (conformally) symplectic manifolds, and since the latter are already well understood in toric geometry,  we will call LCS manifolds only the former.

General Hamiltonian group actions and the corresponding reduction procedure in the LCS and LCK context have been considered by S. Haller and T. Rybicki in \cite{hr}, by R. Gini, L. Ornea and M. Parton in \cite{gop}, or by A. Otiman in \cite{o}. But only recently were Hamiltonian actions of maximal tori on LCK manifolds studied towards a classification, by M. Pilca in \cite{p} and by F. Madani, A. Moroianu and M. Pilca in \cite{mmp}. The program is as follows: there exists a class of LCK manifolds, called Vaisman manifolds, which is better understood via its many geometric properties. In particular, the universal cover of a Vaisman manifold is a \Ka\ cone over a Sasaki manifold. In \cite{p}, toric Vaisman manifolds are studied and it is shown that for every known existing equivalence of categories between them and some other class of manifolds, the Hamiltonian toric action also is equivalent to a natural Hamiltonian toric action in the given category. Then, in \cite{mmp}, it is shown that the corresponding toric Sasaki manifold to a toric Vaisman manifold is actually compact. But Sasaki manifolds are in particular contact, and compact toric contact manifolds have been classified by E. Lerman in \cite{l}. 

On the other hand, in \cite{mmp} toric LCK manifolds of complex dimension 2 have been given a classification, and it turns out that they all admit toric Vaisman metrics. Hence the question was raised of whether this is always the case, regardless of dimension. The main result of this paper is an affirmative answer to it, and so, together with the above cited papers, amounts to a classification of toric LCK manifolds as complex manifolds with a torus action. 

\begin{theoremA}\label{TheoremAA}
Let $(M,J,\Omega)$ be a compact LCK manifold that admits an effective holomorphic twisted Hamiltonian action of a torus of maximal dimension. Then there exists a (possibly different) LCK form $\Omega'$ with respect to which the same action is still twisted Hamiltonian, and such that the corresponding metric $g'$ is Vaisman. 
\end{theoremA}

Remark that the universal cover of a LCK manifold is a non-compact \Ka\ manifold, so one might want to use the theory of toric symplectic manifolds in order to prove the result. However, in the non-compact world the theorems of convexity and connectedness for moment maps of Atiyah and Guillemin-Sternberg fail, and one no longer has a characterisation of the symplectic manifold in terms of the image of the moment map. As proven by E. Lerman and S. Tolman in \cite{kl}, classification results still are possible, but in terms of more complicated objects. Hence we chose to give a direct proof, not relying on the known facts from toric symplectic geometry.

The proof occupies Section \ref{proofA} and roughly goes as follows. First we remark that the holomorphic action of the compact torus $\TT$ on the manifold $M$ naturally extends to a holomorphic action of the complexified torus $\TT^c$. In particular, on the minimal \Ka\ cover $\hat M$ of $M$, $\TT^c$ has a dense connected open orbit, since the $\TT$-action is Hamiltonian. This allows us to view the deck group $\Gamma$ of $\hat M$ as a subgroup of $\TT^c$, and to extend it to a one-parameter subgroup of $\TT^c$. However, there is no reason for this group to act conformally on the LCK form, so at this point we have to construct, by averaging, a new LCK form, still compatible with the $\TT$-action. Finally, we are able to explicitly write down a toric Vaisman metric in the conformal class of the averaged metric. 

The rest of the paper is organised as follows: in Section 2 we speak briefly of equivalent definitions of LCS/LCK manifolds and of properties of their infinitesimal automorphisms. In Section 3 we introduce Hamiltonian group actions in the LCS context. Section 4 puts together the results we use for our proof. In particular, we give a characterisation of when the action of a compact Lie group on a LCS manifold lifts to the minimal symplectic cover, and then show that the maximal dimension of a torus acting effectively on a $2n$-dimensional LCS manifold is $n+1$.

Finally, it is known that compact toric symplectic manifolds are actually toric \Ka, as a consequence of the Delzant classification. It is natural to ask if the analogous fact holds in our setting. It turns out that the answer is negative: in Section 6 we show, by exhibiting an example, that the class of compact toric LCS manifolds strictly contains the compact toric LCK manifolds.

\medskip
\textbf{Acknowledgements.} I thank my PhD advisor Andrei Moroianu for introducing me to the problem and for all the useful suggestions.

\section{Preliminaries on LCS and LCK Geometry}

\subsection{Definitions and Notations}

Let $M$ be a differential manifold.

\begin{definition}
A non-degenerate 2-form $\Omega\in\Omega^2_M$ is called \textit{locally conformally symplectic} or \textit{LCS} if there exists a closed non-exact 1-form $\theta\in\Omega^1_M$, called the \textit{Lee form}, such that:
\begin{equation}\label{Lee}
d\Omega=\theta\wedge\Omega.\end{equation}
Moreover, if there exists an integrable complex structure $J$ with respect to which $\Omega$ becomes a positive $(1,1)$-form, i.e. such that $g(\cdot,\cdot):=\Omega(\cdot,J\cdot)$ is a $J$-invariant Riemannian metric, then $\Omega$ is caled \textit{locally conformally \Ka} or \textit{LCK}. 
\end{definition} 

Equivalently, $\Omega$ is LCS (or LCK respectively) if and only if there exists a covering of $M$ with open sets $\{U_\al\}_{\al\in I}$ such that $\Omega$ restricted to each of them is conformal to a symplectic (respectively \Ka) form:
\[\Omega|_{U_\al}=e^{\phi_\al}\Omega_\al \text{ with } d\Omega_\al=0 \]
where $\phi_\al\in\ce(U_\al)$ and $\Omega_\al\in\Omega^2_M(U_\al)$. Moreover, we have $\theta|_{U_\al}=d\phi_\al$ and the fact that $\theta$ is not (globally) exact corresponds to $(M,\Omega)$ being non conformal to a symplectic (or \Ka) manifold. Also, it can be seen directly that if a form $\Omega$ is LCS or LCK, then for any $u\in\ce(M)$, the form $\Omega_u:=e^u\Omega$ is also LCS or LCK with corresponding Lee form $\theta_u=\theta+du$, thus these notions are conformal in essence. We will denote by $[\Omega]:=\{\Omega_u|u\in\ce(M)\}$ the conformal class of $\Omega$.

\begin{remark} 
This notion is only interesting for us on manifolds of real dimension at least 4. Indeed, on manifolds of dimension 2, any  2-form is automatically closed, hence LCS forms coincide with symplectic forms. On the other hand, in dimension greater than 2, any 1-form $\theta$ verifying \eqref{Lee} is uniquely determined by the LCS form $\Omega$.
\end{remark}

\begin{definition}
A differential manifold endowed with a conformal class of a LCS form $(M, [\Omega])$ is called a \textit{locally conformally symplectic} or \textit{LCS manifod}. A complex manifold endowed with a conformal class of a compatible LCK form $(M,J,[\Omega])$ is called a \textit{locally conformally \Ka\ } or \textit{LCK manifold}. 
\end{definition}

\begin{remark}
Our definition of a LCS/LCK form corresponds to what is usually called a \textit{strict LCS/LCK form}, and for the classical definition, the class of LCS/LCK forms contains the globally conformally symplectic/\Ka\ forms. However, from certain points of view, the category of conformally symplectic manifolds and that of (strict) LCS manifolds behave differently, and for our purposes we are only interested in the second one. In particular, in the complex setting, there is a theorem of I.Vaisman (\cite{va0}) stating that a LCK form on a complex manifold (M,J) is globally conformally \Ka\ if and only if the manifold admits some \Ka\ metric. Thus for our definition, the class of LCK manifolds is disjoint form the class of \Ka ian manifolds.   
\end{remark}

Let $(M,[\Omega])$ be a LCS manifold. On the universal cover $\pi_{\tilde{M}}:\tilde{M}\rightarrow M$, $\pi_{\tilde{M}}^*\theta=d\phi$ is exact and hence $\Omega_0:=e^{-\phi}\pi_{\tilde{M}}^*\Omega\in\Omega^2_{\tilde M}$ is a symplectic form. Since $\pi_{\tilde{M}}^*\theta$ is $\pi_1(M)$-invariant, we have that $\gamma^*\phi-\phi$ is constant for any $\gamma\in\pi_1(M)$, hence we have a group morphism 
\begin{align*}
\rho:\pi_1(M)&\rightarrow(\RR,+)\\
 \gamma&\mapsto\gamma^*\phi-\phi.\end{align*}
 We also have that $\pi_1(M)$ acts on $\Omega_0$ by homotheties:
 \[\gamma^*\Omega_0=e^{-\rho(\gamma)}\Omega_0.\]
 The data $(\tilde M, \pi_1(M), \rho, \Omega_0)$ completely determines the LCS manifold $(M, [\Omega])$. Actually, in order to get a symplectic manifold, one need not consider the universal cover of $M$, but only its minimal cover with respect to which $\theta$ becomes exact. This is precisely $\hat{M}:=\tilde M/\K\rho$ and its deck group over $M$ is $\Gamma=\pi_1(M)/\K\rho\cong \I\rho$, which is a free abelian subgroup of $(\RR,+)$, hence isomorphic to $\ZZ^k$ for some $k$. Of course, the sympectic form of $\tilde M$ descends to $\hat M$, and we will denote it also by $\Omega_0$. 
 
 Note that the de Rham class of the Lee form $[\theta]_{dR}\in H^1_{dR}(M,\RR)$ is invariant under a conformal change of the LCS form $\Omega$, hence also the minimal cover $\hat M$ with the corresponding symplectic form $\Omega_0$, up to multiplication by positive constants, depends only on the conformal class $[\Omega]$. Also note that if $\Omega$ is a LCK form with respect to a complex structure $J$, then the sympectic form $\Omega_0$ is \Ka\ with respect to the pull-back complex structure. 

For a given LCS form $\Omega$ with its Lee form $\theta$, we have the twisted differential $d^\theta:=d-\theta\wedge\cdot$ which verifies $d^\theta\Omega=0$ and $d^\theta\circ d^\theta=0$. There is a result specific to the LCS geometry which is not valid in conformal symplectic geometry, that will be useful to us:
\begin{lemma}(\cite{va1},\cite{mmp})\label{injective}
 The map $d^\theta:\ce(M)\rightarrow \Omega^1(M)$ is injective on a LCS manifold $(M,\Omega)$.
\end{lemma}

 In LCK geometry, there is a special class of manifolds that behaves particularly nicely, characterised by a metric property:
\begin{definition}
 Let $(M,J,\Omega)$ be a LCK manifold with corresponding metric $g$ and Lee form $\theta$. The metric $g$ is called \textit{Vaisman} if the form $\theta$ is parallel with respect to the Levi-Civita connection corresponding to $g$.
\end{definition}

Vaisman manifolds are closely related to Sasaki manifolds: the universal cover of a Vasiman manifold with its \Ka\ metric is isometric to the \Ka\ cone over a Sasaki manifold. We recall that a Sasaki manifold $(S,g_S,\tilde J)$ is a Riemannian manifold $(S,g_S)$ together with a complex structure $\tilde J$ on $S\times \RR$ with respect to which the cone metric $g_K:=e^{-2t}(g_S+dt^2)$ is \Ka, and the homotheties $\psi_s(z,t)=(z,t+s)$, $s\in\RR$ are holomorphic.

We will not insist at all on the properties of Vaisman manifolds, but invite the reader to consult \cite{do} and the references therein. 

 \begin{remark}
 From now on, for a LCS manifod $(M,\Omega)$ we will always use the notations $\theta, \hat M, \Gamma, \rho, \phi$ or $\Omega_0$ to denote the uniquely associated objects as seen above, without redefining them. Also, we will always suppose, unless otherwise stated, that the LCS/LCK manifolds are compact connected of real dimension at least 4.
 \end{remark}

\subsection{Infinitesimal Automorphisms}\label{sec}

In this section we will take a closer look at the Lie algebra of infinitesimal automorphisms of LCS or LCK manifolds, and will distinguish a special subalgebra that will play a particular role.
 
For a LCS manifold $(M, [\Omega])$, the automorphism group $\Aut(M,[\Omega])$ is formed by all the conformal diffeomorphisms $\Phi:M\rightarrow M$, $\Phi^*\Omega\in[\Omega]$. By obvious analogy, for a LCK manifold $(M, J, [\Omega])$, the automorphism group $\Aut(M, J, [\Omega])$ is given by all the conformal biholomorphisms. Denote by $\aut(M, [\Omega])$ and by $\aut(M, J, [\Omega])$ respectively the corresponding Lie algebras of infinitesimal automorphisms. 

First of all, note that $X\in \aut(M, [\Omega])$ means $\LL_X\Omega=f_X\Omega$. This implies $(f_X-\theta(X))\Omega=d^\theta(\iota_X\Omega)$. Hence $d^\theta((f_X-\theta(X))\Omega)=0$, or also $(df_X-d(\theta(X)))\wedge\Omega=0$ and since we are working under the supposition that $\dim M\geq 4$, it follows that $\theta(X)-f_X=c_X\in\RR$. By straightforward computations it can be seen that the constants $c_X$ are conformally invariant. Hence we have a linear map: 
\begin{align}\label{mapl}
\begin{split}
l:\aut(M,[\Omega])&\rightarrow \RR \\
X&\mapsto c_X=\theta(X)-f_X\end{split}\end{align}
which actually can be seen to be a Lie algebra morphism (see \cite{va1} for all the details). Consider the kernel of this map, which is also conformally invariant:
\[\aut'(M,[\Omega]):=\{X\in\Gamma(TM)|\LL_X\Omega=\theta(X)\Omega\}.\]
We will call elements of this subalgebra horizontal or special conformal vector fields. 

On the other hand, for $X\in\Gamma(TM)$ and $\hat X:=\pi^*X\in\Gamma(T\hat M)$ its lift to $\hat M$, we have the following formula:
\begin{equation}\label{liftLie}
\LL_{\hat X}\Omega_0=e^{-\phi}\pi^*(\LL_X\Omega-\theta(X)\Omega).
\end{equation}
In particular, for $X\in\aut(M,[\Omega])$ we have $\LL_{\hat X}\Omega_0=-l(X)\Omega_0$.

Also, under the hypothesis $\dim M \geq 4$, every conformal automorphism of the symplectic form $\Omega_0$ is in fact a homothety:
\[\LL_X\Omega_0=f\Omega_0 \Rightarrow 0=d(\LL_X\Omega_0)=df\wedge\Omega_0\Rightarrow df=0. \]
Hence we have proved the following lemma, which emphasises the role of $\aut'(M,[\Omega])$:

\begin{lemma}
We have a natural isomorphism between  $\aut(M,[\Omega])$ and $\aut(\hat M,[\Omega_0])^\Gamma$ given by $\pi^*$. In particular, under this isomorphism, $\aut'(M,[\Omega])$ is in bijection with $\aut(\hat M, \Omega_0)^\Gamma$, the Lie algebra of $\Gamma$-invariant infinitesimal symplectomorphisms of $\Omega_0$.
\end{lemma}

\begin{definition}
A LCS form $(\Omega,\theta)$ is called \textit{exact} if $\Omega$ is $d^\theta$ exact, i.e. $\Omega=d^\theta\eta$, for some form $\eta\in\Omega^1_M$. Remark that, in this case, for any other LCS form in the same conformal class $\Omega_u=e^u\Omega$ with Lee form $\theta_u=\theta+du$ we have $\Omega_u=d^{\theta_u}(e^u\eta)$. Hence we can call an LCS manifold $(M,[\Omega])$ \textit{LCS exact} if some, and hence any representative $\Omega$ is exact.
\end{definition}

The map $l$ defined in \eqref{mapl} is studied by I.Vaisman in \cite{va1}, and in particular its restriction to $\aut(M, \Omega)$. He names LCS manifolds of the first kind those for which this restriction is not identically zero and studies their structure. As noted in \cite{va1}, being of the first kind is not a conformally invariant notion. However, we have the following result in the conformal setting:

\begin{lemma}\label{firstkind}
The map $l$ is surjective iff $(M,[\Omega])$ is LCS exact.
\end{lemma}

\begin{proof}
First of all fix $\Omega\in[\Omega]$ a LCS form. Suppose $l\not\equiv 0$ and choose $B\in\aut(M,[\Omega])$ such that $l(B)=1$. Then we have: 
\begin{align*}
\theta(B)\Omega-\Omega=\LL_B\Omega=d\iota_B\Omega+\theta(B)\Omega-\theta\wedge \iota_B\Omega\end{align*}
hence:
\[\Omega=d^\theta(-\iota_B\Omega).\]
Conversely, suppose $\Omega=d\eta-\theta\wedge\eta$. Define $B\in\Gamma(TM)$ by: $\iota_B\Omega=-\eta$. We compute:
\begin{align*}
\LL_B\Omega=&d\iota_Bd\eta-d(\theta(B)\eta)+\iota_B(\theta\wedge d\eta)=\\
=&d(\theta(B)\eta-\eta)-d(\theta(B)\eta)+\theta(B)d\eta-\theta\wedge(\theta(B)\eta-\eta)=\\
=&(\theta(B)-1)\Omega.
\end{align*}
Hence $B\in\aut(M,[\Omega])$ and $l(B)=1$.
\end{proof}

\section{Twisted Hamiltonian Vector Fields}

In this section, we study the corresponding notions of Hamiltonian vector field and Hamiltonian group action to the LCS context. The definitions, as presented, were introduced by I.Vaisman in \cite{va1}, where one can also see a number of reasons for why these are the natural analogues to the ones from the symplectic world.

\begin{definition}
A vector field $X\in\Gamma(TM)$ on a LCS manifold $(M,\Omega,\theta)$ is called \textit{twisted Hamiltonian} if there exits a function $f\in\ce(M)$ such that $\iota_X\Omega=d^\theta f$.
\end{definition}

\begin{remark}
Although it is not apparent from the definition, the above notion is actually conformally invariant. Indeed, if $X=X_f$ is a twisted Hamiltonian vector field for $\Omega$ with corresponding function $f\in\ce(M)$ and $\Omega':=e^u\Omega$ is another conformal form with corresponding Lee form $\theta'=\theta+du$, then we have:
\begin{equation}\label{hamconf}
\iota_X\Omega'=e^u(df-\theta f)=d^{\theta'}(fe^u).
\end{equation}
\end{remark}

\begin{remark}\label{Poisson} As in the symplectic setting, an LCS form $\Omega$ defines on $\ce(M)$ a Poisson bracket:
\[ \{f,g\}:=\Omega(X_g,X_f)  \quad \forall f,g\in\ce(M) \]
and by straightforward calculations it can be seen that $X_{\{f,g\}}=[X_f,X_g]$. Hence the set of twisted Hamiltonian vector fields $\ham(M, [\Omega]):=\{X\in\Gamma(TM)| \exists f\in\ce(M) \ \iota_X\Omega=d^\theta f\}$ forms a Lie subalgebra of $\Gamma(TM)$. 
\end{remark} 

\begin{remark}
Actually, $\ham(M, [\Omega])\subset\aut'(M,[\Omega])$. Indeed, for $X=X_f\in\ham(M, [\Omega])$ we have:
\begin{align}\label{cia}
\LL_{X_f}\Omega=\iota_{X_f}d\Omega+d\iota_{X_f}\Omega=\iota_{X_f}(\theta\wedge\Omega)+d(df-\theta f)=\theta(X_f)\Omega.
\end{align}
\end{remark}

\begin{remark}
The pull-back morphism $\pi^*$ establishes an injection between $\ham(M, [\Omega])$ and the Lie algebra of Hamiltonian vector fields of the symplectic form on the minimal cover $\ham(\hat M,\Omega_0)$. Indeed, if $X=X_f\in\ham(M, \Omega)$ and $\hat X=\pi^* X$ is the pull-back vector field to $\hat M$, by writing $\Omega_0=e^{-\phi}\pi^*\Omega$ we have:
\begin{align*}
\iota_{\hat X}\Omega_0=e^{-\phi}(d\pi^*f-\pi^*fd\phi)=d(e^{-\phi}\pi^*f).
\end{align*}
\end{remark}

For $\Omega\in[\Omega]$, define the map $A^\Omega:\ce(M)\rightarrow \ham(M, [\Omega])$ by sending a function $f$ to its corresponding Hamiltonian vector field $X_f$ with respect to $\Omega$. By Remark \ref{Poisson}, if we consider on $\ce(M)$ the Lie algebra structure given by $\Omega$, $A^\Omega$ is a Lie algebra morphism. Note that since $d^\theta:\ce(M)\rightarrow \Omega^1_M$ is injective by Lemma \ref{injective}, also $A^\Omega$ is, hence $A^\Omega$ is  actually an isomorphism of Lie algebras. By \eqref{hamconf} we have $A^{e^u\Omega}(f)=A^\Omega(e^uf)$. 

\begin{definition}
Let $(M,[\Omega])$ be a LCS manifold. We say that an action of a Lie group $G$ is \textit{twisted Hamiltonian} if $\Lg:=\Lie(G)\subset\ham(M, [\Omega])$. 
\end{definition}

\begin{remark}\label{reference1}
If the Lie group $G$ is compact and acts conformally on $[\Omega]$, then we can find a LCS form in the given conformal class that is $G$-invariant. Indeed, take any LCS form $\Omega\in[\Omega]$. Then, for any $g$ in $G$, we have $g^*\Omega=e^{f_g}\Omega$, with $f_g\in\ce(M)$. Let $dv$ be a normalised Haar measure on $G$, and take $h:=\int_Gf_gdv(g)$, so that $\Omega^G:=\int_Gg^*\Omega dv(g)=e^h\Omega$. Then $\Omega^G\in[\Omega]$ is, by definition, a $G$-invariant LCS form with corresponding Lee form $\theta^G=dh+\theta$.
\end{remark}

Suppose that a compact Lie group $G$ has a twisted Hamiltonian action on the LCS manifold $(M, [\Omega])$. As soon as we choose a LCS form $\Omega\in[\Omega]$, there automatically exists a moment map, that is a Lie algebra morphism $\mu^\Omega:\Lg\rightarrow \ce(M)$ which is a section of $A^\Omega$. More precisely, for any $X\in\Lg$, if we denote by $\mu^\Omega_X:=\mu^\Omega(X)\in\ce(M)$, we have $\iota_X\Omega=d^\theta\mu^\Omega_X$. This comes from the fact that $A^\Omega$ is an isomorphism of Lie algebras, hence $\mu^\Omega=(A^\Omega)^{-1}|_\Lg$.

\begin{remark}
Similarly to the symplectic context, $\mu^\Omega$ can also be seen as an application $\mu^\Omega:M\rightarrow \Lg^*$ via $\langle \mu^\Omega,X\rangle=\mu^\Omega_X$. However, the fact that the first application is a Lie algebra morphism does not automatically imply that the second one is $G$-equivariant. In fact, being a Lie algebra morphism is conformally invariant, while being equivariant is not. Nonetheless, if $G$ is abelian, the equivariance of $\mu^\Omega$ is equivalent to $\Lg\subset\ker \theta$, or also to $\Omega$ being $G$-invariant.
\end{remark}

\begin{remark}\label{exacthamiltonian}
If $(M,[\Omega])$ is an exact LCS manifold, and $G$ is a compact Lie group that acts conformally on it such that $\Lg=\Lie(G)\subset\aut'(M,[\Omega])$, then this action is automatically twisted Hamiltonian. Indeed, as before, we can choose from the beginning, in the given conformal class, $\Omega=d^\theta\eta$ and $\theta$ $G$-invariant. Now define $\eta^G:=\int_Gg^*\eta dv(g)$. Then, since $\theta$ is $G$-invariant, we have:
\[d^\theta\eta^G:=\int_Gg^*(d\eta)dv(g)-\theta\wedge\int_Gg^*\eta dv(g)=\int_Gg^*(d\eta-\theta\wedge\eta)dv(g)=\int_G\Omega dv(g)=\Omega.\]
Hence, there exists a momentum map given by the $G$-invariant form $-\eta^G$. More precisely, we have, for any $X\in\Lg$:
\[\iota_X\Omega=\iota_Xd\eta^G+\theta\eta^G(X)=\LL_X\eta^G-d(\eta^G(X))+\theta\eta^G(X)=d^\theta(-\eta^G(X)).\] 
\end{remark}

\section{Torus actions on LCS manifolds}

In this section we assemble mostly already known results concerning tori actions that we will need in the sequel. In particular we will make use of the following well-known general result about the orbits of smooth actions of compact Lie groups, which is a consequence of the slice theorem. We refer the reader to \cite{br} or to \cite{dk} for a proof of the result and for a detailed presentation of the subject.

\begin{theorem}\label{maximalorbits}
Let $N$ be a connected smooth manifold and $G$ be a compact Lie group which acts effectively by diffeomorphisms on $N$. For any $x\in N$, denote by $G_x:=\{g\in G|g.x=x\}$ the stabiliser of $x$ in $G$, and let $r=\inf_{x\in N}\dim G_x$. Then $N_r:=\{x\in N| \dim G_x=r\}$, called the set of principal $G$-orbits, is a dense connected open submanifold of $N$, and $N-N_r$ is a union of submanifolds of codimension $\geq 2$. Moreover, if $G$ is abelian and acts effectively on $N$, then $r=0$. \end{theorem}

\begin{remark}\label{freeM0}
In general, if $G=\TT$ is the compact torus and acts effectively on $N$ as in the above theorem, the stabilisers $G_x$ need not be connected. However, if in addition we have a symplectic form on $N$ which is preserved by $G$ and such that the orbits of the action are isotropic, then indeed all the stabilisers are connected tori. For a proof of this, see for instance \cite[Lemma~6.7]{be}. In particular, the set $N_0$ is acted upon freely. As we will see soon, cf. Proposition \ref{maxdim}, this hypothesis will be verified in our context.
\end{remark}

Since we will need to switch between the compact LCK manifold and the non-compact \Ka\ covering,  we need to know what happens with a given torus action in the process. Its behaviour actually does not involve the complex structure of the manifold, so next we give an analogue of \cite[Proposition~4.4]{mmp} in the LCS setting, when the group is a torus. The above result can be shown to hold for any compact Lie group, following the arguments of \cite{mmp} and using the structure theorem of compact Lie groups.

\begin{proposition}\label{liftaction}
Let $(M,[\Omega])$ be a LCS manifold and $\TT$ be a compact torus acting on $M$ by conformal automorphisms. Then the action of $\TT$ lifts to the minimal cover $\hat M$ iff $\Lie(\TT)=\Lt\subset\aut'(M,[\Omega])$.
\end{proposition}
\proof
 
We can suppose that $\TT=S^1$, for otherwise we make use of the same argument for each generator of the $\TT$-action. Fix a LCS form $\Omega$. Denote by $X$ the generator of the infinitesimal action of $S^1$ on $M$, by $\hat X$ its lift to $\hat M$ and by $\Phi_t$ and $\hat\Phi_t$ their corresponding flows, so that $\Phi_0=\Phi_1=\id_M$. Then the action of $\TT$ lifts to $\hat M$ iff $\hat\Phi_t$ is periodic in $t$.
 
Suppose first that $\Lt\subset\aut'(M,[\Omega])$. Equation \eqref{liftLie} implies then that $\hat X\in\aut(\hat M,\Omega_0)$, hence $\{\hat\Phi_t\}_t$ are symplectomorphisms. On the other hand, $\hat\Phi_1$ is an element of $\Gamma$ since it covers the identity of $M$. Thus $\hat\Phi_1\in\K\rho=\{\id\}$ by the definition of the minimal cover.

Conversely, suppose $\hat\Phi_t$ is periodic in $t$. As we already saw, $\hat X$ acts by homotheties on $\Omega_0$, hence the exists a periodic $\ce$ function $c:\RR\rightarrow \RR$ such that $\hat\Phi_t^*\Omega_0=c(t)\Omega_0$. Moreover, we have, for any $t_1, t_2\in\RR$: 
\[c(t_1+t_2)\Omega_0=\hat\Phi_{t_1}^*(\hat\Phi_{t_2}^*\Omega_0)=\hat\Phi_{t_1}^*(c(t_2))\hat\Phi_{t_1}^*\Omega_0=c(t_2)c(t_1)\Omega_0. \]
Hence, for any $t\in\RR$:
\[\dot c(t)=\lim_{h\to 0}\frac{c(t+h)-c(t)}{h}=\lim_{h\to 0}\frac{c(t)c(h)-c(t)c(0)}{h}=c(t)\dot c(0).\]
On the other hand, since $c$ is periodic, it must have some critical point, implying that $\dot c(0)=0$. Therefore $\LL_{\hat X}\Omega_0=\dot c(0)\Omega_0=0$, or also, by \eqref{liftLie}, $X\in\aut'(M,[\Omega])$. 
\endproof

\begin{remark}\label{liftedaction}
Note that, in general, an action of a group $G$ on $\hat M$ descends to an action of $G$ on $M$ iff $G$ commutes with $\Gamma$. \end{remark}

\begin{corollary}
Any twisted Hamiltonian action of a compact torus $\TT$ on a LCS manifold $(M,[\Omega])$ lifts to a Hamiltonian action of $\TT$ to the minimal symplectic cover $(\hat M, \Omega_0)$.
\end{corollary}

\begin{proof}
Indeed, by \eqref{cia}, $\Lt$ sits in $\aut'(M, [\Omega])$, so the $\TT$-action lifts to $\hat M$. Moreover, the lifted action is still Hamiltonian, since it admits the moment map $\hat \mu:\hat M\rightarrow \Lt^*$, $\hat \mu(\hat x)=e^{-\phi(\hat x)}\mu^\Omega(\pi(\hat x))$. Remark that we chose a form $\Omega\in[\Omega]$ in order to define $\hat\mu$, but actually $\hat\mu$ is conformally invariant. 
\end{proof}

Recall that, for a symplectic manifold, the maximal dimension of a torus acting symplectically and effectively on it is bounded from above only by the dimension of the manifold and, moreover, in many cases the orbits are not isotropic. The next proposition shows that things are different in the LCS setting. A variant of this result can again be found as Proposition 3.9 in \cite{mmp}. 

\begin{proposition}\label{maxdim}
Suppose that a real torus $\TT^m$ acts conformally end effectively on a LCS manifold $(M^{2n},[\Omega])$. Then $m\leq n+1$ and, moreover, if $\Lt=\Lie(\TT^m)\subset\aut'(M,[\Omega])$, then $m\leq n$ and the orbits are isotropic with respect to any representative in $[\Omega]$.
\end{proposition}
\begin{proof}
Denote by $\T\subset TM$ the distribution generated by $\TT^m$ on $M$, and by $\hat \T$ the one on $\hat M$. Suppose first that $\Lt\subset\aut'(M,[\Omega])$. By \eqref{liftLie} it follows that $\Gamma(\hat \T)\subset\aut(\hat M,\Omega_0)$. Hence, using the formula: 
\begin{equation}\label{formula}
\iota_{[X,Y]}=\LL_X\iota_Y-\iota_Y\LL_X\end{equation}
we have, for any $\hat X$ and $\hat Y$ in $\Gamma(\hat \T)$:
\begin{align*}
0=\iota_{[\hat X,\hat Y]}\Omega_0=\LL_{\hat X}\iota_{\hat Y}\Omega_0=d\iota_{\hat X}\iota_{\hat Y}\Omega_0+\iota_{\hat X}d\iota_{\hat Y}\Omega_0.
\end{align*}
But we also have:
\[d\iota_{\hat Y}\Omega_0=\LL_{\hat Y}\Omega_0=0\]
implying that $d(\Omega_0(\hat X,\hat Y))=0$, or also that $\Omega_0(\hat X,\hat Y)=c\in\RR$. It follows that $e^\phi c =\pi^*(\Omega(X,Y))$, and since $e^\phi$ is not $\Gamma$-invariant, $c=0$. Therefore, for any $\hat x\in\hat M$, $\hat \T_{\hat x}$ is isotropic with respect to $(\Omega_0)_{\hat x}$, so $\hat \T$ and also $\T$ have maximal rank at most $n$.

On the other hand, let $M_0\subset M$ be the dense open set composed by all the $m$-dimensional orbits, as in Theorem \ref{maximalorbits}.  Then $M_0\times \Lt$ injects into $\T|_{M_0}$ as a vector subbundle in a natural way, hence $m\leq n$. 

In the general case, if $\Lt\not\subset\aut'(M,[\Omega])$, then by \eqref{mapl} there exists $B\in\Lt-\aut'(M,[\Omega])$ such that $l(B)=1$. Then we have a splitting $\Lt=\RR B\oplus\Lt'$ with $\Lt'\subset\aut'(M,[\Omega])$ and by the above, $\Lt'$ has dimension at most $n$, hence the conclusion follows.\end{proof}

\begin{definition}
A LCS manifold $(M^{2n},[\Omega])$ is called \textit{ toric LCS} if the maximal compact torus $\TT^n$ acts effectively in a twisted Hamiltonian way on it. A LCK manifold $(M^{2n}, J, [\Omega])$ is called \textit{toric LCK} if $(M^{2n}, [\Omega])$ is toric LCS with respect to an action of $\TT^n$ which is moreover holomorphic.
\end{definition}

The first examples of toric LCS/LCK manifolds are given by the diagonal Hopf surfaces $\CC^2-\{0\}/\Gamma$ with the Vaisman metrics constructed in \cite{go} and with the standard torus action. For a detailed proof, see \cite{p} where toric LCK manifolds were first considered. See also \cite{mmp} for another construction of toric LCK manifolds out of toric Hodge manifolds.

\section{Proof of Theorem A}\label{proofA}

We are now ready to give the proof of the main result:

\begin{theoremA}
Let $(M,J,[\Omega])$ be a compact toric LCK manifold. Then there exists a LCK form $\Omega'$ (possibly nonconformal to $\Omega$) with respect to which the same action is still twisted Hamiltonian, and such that the corresponding metric $g'$ is Vaisman. 
\end{theoremA}

\begin{proof}

Denote by $\TT$ the $n$-dimensional compact torus that acts on the LCK manifold as in the hypotheses of the theorem. Then the holomorphic action of $\TT$ naturally extends to an effective holomorphic action of the complexified torus $\TT^c=(\CC^*)^n$ on $M$. Indeed, on one hand the induced inclusion homomorphism $\tau:\Lt\rightarrow \aut(M, J)$ extends to a Lie algebra morphism $\tau:\Lie(\TT^c)=\Lt\otimes\CC\rightarrow\aut(M, J)$ by: 
\[\tau(\xi_1+i\xi_2)=\tau(\xi_1)+J\tau(\xi_2)=X_{\xi_1}+ JX_{\xi_2}.\] 
On the other hand, we have the Cartan decomposition $\TT^c=\TT\times (\RR_{>0})^n$ and $i\Lt\subset \Lt^c$ is isomorphic to $(\RR_{>0})^n$ under the exponential map. Hence, if $\xi_1,\ldots , \xi_n$ form a basis of the Lie algebra $\Lt$, then $JX_{\xi_1},\ldots , JX_{\xi_n}\in\aut(M,J)$, being complete vector fields, generate the effective holomorphic action of $(\RR_{>0})^n$ on $M$. 

Let $(\hat M, \hat J, \Omega_0)$ be the minimal \Ka\ cover of $(M, J, [\Omega])$ of deck group $\Gamma$.  The action of $\TT^c$ evidently lifts to $\hat M$, and $\TT$ also acts in a Hamiltonian way with respect to $\Omega_0$. Denote by $\hat\mu:\hat M\rightarrow \RR^n$ the moment map of this action, and let $\hat M_0\subset \hat M$ be the corresponding connected dense open set of principal $\TT$-orbits, as in Theorem \ref{maximalorbits}. Following the proof of Proposition \ref{maxdim}, the orbits of $\TT$ on $\hat M$ are isotropic, hence cf. Remark \ref{freeM0}, $\hat M_0$ coincides with the set of points of $\hat M$ on which $\TT$ acts freely.

\underline{Fact 1}: $\TT^c$ preserves $\hat M_0$ and acts freely on it.

By the above, $\TT^c=\TT\times(\RR_{>0})^n$ preserves $\hat M_0$ iff $\forall u\in (\RR_{>0})^n,$ $\forall \hat x\in\hat M_0$, $\forall t\in \TT-\{1\}$, $tu.\hat x\neq u.\hat x$. But this is obvious since $t$ and $u$  commute.

To show that the action of $\TT^c$ is free on $\hat M_0$, let $g\in \TT^c$ and $\hat x\in\hat M_0$ with $g.\hat x=\hat x$. With the above remarks on $\TT^c$, we have $g=tu$ with $t\in \TT$ and $u=\exp(i\xi),$ $\xi\in\Lt$. By letting $\hat y:=t.\hat x$, it follows that $u.\hat y=\hat x\in \TT\hat x=\TT\hat y$. Let $c:\RR\rightarrow \hat M$ be the curve $c(s)=\exp(is\xi).\hat y$. Since $\hat\mu^\xi$ is constant on the orbits of $\TT$, it follows that:
\begin{equation}\label{miu}
\hat\mu^\xi(c(0))=\hat\mu^\xi (\hat y)=\hat\mu^\xi (\hat x)=\hat\mu^\xi (c(1)).\end{equation} 
On the other hand, the vector field $\tau(i\xi)=JX_\xi$ is, by definition, the gradient of the Hamiltonian $\hat\mu^\xi$. So, if $\xi\neq 0$, then $\hat\mu^\xi$ would be strictly increasing along $c$, but this contradicts \eqref{miu}. Thus $\xi=0$ and we have $t.\hat x=\hat x$, implying again that $t$ is the trivial element in $\TT$, hence $g$ is the trivial element in $\TT^c$.

\underline{Fact 2}: $\TT^c$ acts  transitvely on $\hat M_0$.

For any $\hat x\in\hat M_0$, the map $\TT^c\rightarrow \hat M_0$, $g\mapsto g.\hat x$ is a holomorphic open embedding. Therefore, the connected open set $\hat M_0$ is a reunion of disjoint open orbits of $\TT^c$, hence it must contain (and be equal to) a sole orbit.

In conclusion, for any choice of a point $\hat x_0\in\hat M_0$ we have a $\TT^c$-equivariant biholomorphism $F_{\hat x_0}:(\CC^*)^n\rightarrow \hat M_0$, $g\mapsto g.\hat x_0$, where $(\CC^*)^n$ acts on itself by (left) multiplication. On the other hand, $\Gamma$ preserves $\hat M_0$, hence we can view $\Gamma$ as a subgroup of biholomorphisms of $(\CC^*)^n$ acting freely. 

\underline{Fact 3}: $\Gamma\subset \TT^c$.

Let $\hat y=F_{\hat x_0}(g)$, with $g\in \TT^c$, be any element of $\hat M_0$ and let $\gamma\in\Gamma$. Denote by $g_\gamma\in \TT^c$ the element verifying $\gamma(\hat x_0)=F_{\hat x_0}(g_\gamma)$. Since the action of $\TT^c$ on $\hat M $ is the lift of the action of $\TT^c$ on $M$, following Remark \ref{liftedaction}, $\Gamma$ commutes with $\TT^c$. We thus have:
\begin{align*}
\gamma(\hat y)=\gamma(g.\hat x_0)=g.\gamma(\hat x_0)=g.g_\gamma.\hat x_0=g_\gamma.\hat y
\end{align*}
implying that $\gamma=g_\gamma\in \TT^c$.

Remark that, if $g_\gamma^j\in\CC^*$ are the components of $g_\gamma$, then for at least one $1\leq j\leq n$, $|g_\gamma^j|\neq 1$. Otherwise we would have $\Gamma\subset \TT$ and so $\TT$ would not act effectively on $M$.

Let now $\gamma\in \Gamma\cong \ZZ^k$ be a nontrivial primitive element and denote by $\Gamma'$ the subgroup generated by $\gamma$. With the same notations as before, we can extend the action of $\Gamma'$ on $\hat M_0\cong(\CC^*)^n$ to a holomorphic action of $\RR$ on $\hat M_0$. Indeed, if $\gamma$ expresses, as an automorphism of $(\CC^*)^n$, as:
\[\gamma(z_1,\ldots,z_n)=(\al_1z_1,\ldots,\al_nz_n),\] 
with $\al_j=\rho_je^{i\theta_j}$ in polar coordinates, then define the one-parameter group: 
\begin{gather*}
\RR\ni t\mapsto \Phi_t\in\Aut((\CC^*)^n)\\
\Phi_t(z_1,\ldots, z_n)=(\rho_1^te^{it\theta_1}z_1,\ldots,\rho_n^te^{it\theta_n}z_n).
\end{gather*} 
Remark that $\RR\cong\{\Phi_t\}_{t\in\RR}$ is a subgroup of $\TT^c\subset\Aut((\CC^*)^n)$, hence its action on $\hat M_0$ actually extends to the whole of $\hat M$. Moreover, this also implies that $\Gamma$ commutes with $\RR$, so the action of $\RR$ descends on $M$ to an effective action of $\RR/\Gamma'\cong S^1$. Let $C\in\Gamma(TM)$ be the real holomorphic vector field generating this action.

\begin{lemma}\label{average2}
There exists on $M$ a LCK form $\Omega^C$ compatible with the complex structure $J$, with corresponding Lee form $\theta^C$, so that $C$ preserves both $\Omega^C$ and $\theta^C$. Moreover, the given action of $\TT$ is still Hamiltonian with respect to this new form.
\end{lemma}

\begin{proof}
For any $t\in\RR$ let $f_t:=\Phi_t^*\phi-\phi\in\ce(\hat M)$ and define $h:=\int_0^1f_tdt\in\ce(\hat M)$. Note that the functions $\{f_t\}_{t\in\RR}$ are $\Gamma$-invariant:
\begin{align}
\delta^*f_t&=\Phi_t^*\delta^*\phi-\delta^*\phi=\Phi_t^*(\phi+\rho(\delta))-(\phi+\rho(\delta))=f_t, \qquad \forall \delta\in\Gamma
\end{align}
hence so is $h$ and they all descend to $M$. Moreover, since $\Lt\subset\ker\theta$, $\phi$ is $\TT$-invariant. As $\TT$ commutes with $\{\Phi_t\}_{t\in\RR}$, it follows that also the function $h$ is $\TT$-invariant.

Let the new Lee form be:
\begin{equation}\label{thetac}
\theta^C:=\int_{\RR/\Gamma'}\Phi_t^*\theta dt=d\int_0^1\Phi_t^*\phi dt=d(\phi+h)=\theta+dh.
\end{equation}
By definition, it is $C$-invariant, but also $\TT$-invariant since $\Lt$ commutes with $C$.

Let now $\Omega_h:=e^h\Omega\in\Omega^2(M)$ and define the new LCK form as:
\begin{equation*}
\Omega^C:=\int_{\RR/\Gamma'}\Phi_t^*\Omega_h dt.
\end{equation*}
Since $d\Omega_h=\theta^C\wedge \Omega_h$ by \eqref{thetac},  we see that the Lee form of $\Omega^C$ is indeed $\theta^C$:
\begin{align*}
d\Omega^C=&\int_{\RR/\Gamma'}\Phi_t^*(d\Omega_h)dt=\int_{\RR/\Gamma'}\Phi_t^*\theta^C\wedge\Phi_t^*\Omega_h dt=\\
=&\theta^C\wedge\int_{\RR/\Gamma'}\Phi_t^*\Omega_h dt=\theta^C\wedge\Omega^C.
\end{align*}
Again, the $C$-invariance of $\Omega^C$ follows from its definition. Moreover, since $h$ is $\TT$-invariant and $\TT$ commutes with $\RR/\Gamma'$, also $\Omega^C$ is $\TT$-invariant.

Finally, $\LL_C\theta^C=0$ implies that $\theta^C(C)$ is constant. On the other hand $\theta^C(C)=\LL_C\phi$, and since $\phi$ is not even $\Gamma$-invariant, it follows that $\theta^C(C)=\la\neq0$. Hence, by Lemma \ref{firstkind}, the form $\eta=-\frac{1}{\la}\iota_{C}\Omega^C\in\Omega^1_M$ verifies $\Omega^C=d^{\theta^C}\eta$. Moreover, $\eta$ is automatically $\TT$-invariant, since both $C$ and $\Omega^C$ are. Therefore, cf. Remark \ref{exacthamiltonian}, we have a moment map for the action of $\TT$ on $(M, J,\Omega^C)$ given by $\mu^C(X)=-\eta(X)$, implying that the action is still Hamiltonian.
\end{proof}

\begin{lemma}
The minimal cover corresponding to the form $\Omega^C$ is $\hat M$.
\end{lemma}
\begin{proof}
Let $p_C:\hat M_C\rightarrow M$ be the minimal \Ka\ cover corresponding to $\Omega^C$ with deck group $\Gamma_C$, and denote by $p:\hat M\rightarrow M$ the projection corresponding to $\Omega$. We have $p\Phi_t=\Phi_tp$ for any $t\in\RR$, by making no distinction of notation between objects on $M$ and on $\hat M$. 
We see, by \eqref{thetac} in Lemma \ref{average2}, that $p^*\theta^C=d\phi^C$ is exact, where $\phi^C=\phi+p^*h$. So  $\hat M$ is a covering of $\hat M_C$ and $\Gamma_C$ is a subgroup of $\Gamma'$. On the other hand, by the same lemma, $h$ is $\Gamma$-invariant, so for any $\delta\in\Gamma'$ we have $\delta^*\phi^C=\rho(\delta)+\phi^C$. Thus no element of $\Gamma'$ preserves $\phi^C$, therefore $\hat M_C=\hat M$.
\end{proof}

We also give here a lemma since it follows directly from the above considerations, but we will not make use of it in the sequel.

\begin{lemma}\label{rankGamma}
The rank of $\Gamma$ is 1.
\end{lemma}
\begin{proof}
With the same notations as before, suppose there existed some $\gamma'\in\Gamma$ independent (over $\ZZ$) of $\gamma$. Then, in the same way, $\gamma'$ would generate another real holomorphic vector field $C'\in\Gamma(TM)$, independent of $C$. Indeed, if this was not the case, then suppose we have $C=aC'$ with $a\in\RR$. Then the corresponding flows would verify $\Phi^t_C=\Phi^t_{aC'}=\Phi^{at}_{C'}$. In particular, for any $m\in\ZZ$ we would have that $\Phi^m_C=\Phi_{C'}^{am}\in\Gamma$. From the independence of $\gamma$ and $\gamma'$ it follows that $a\not\in\QQ$, so the additive subgroup $\Lambda$ generated by $1$ and $a$ in $(\RR,+)$ is not discrete. Now, if we fix some $\hat x\in\hat M$, the map $F:\RR\rightarrow \hat M$, $t\mapsto \Phi^t_{C'}(\hat x)$ is continuous, so also $F(\Lambda)\subset \hat M$ is not discrete. But $F(\Lambda)$ is contained in the fiber of the covering map through $\hat x$ which must be discrete, hence we have a contradiction.

Now let $\Omega'$ be the LCK form obtained by averaging $\Omega^C$ with respect to $C'$, as in Lemma \ref{average2}. We would thus have an effective holomorphic action of $\TT^{n+2}$ on $
M$ generated by $\Lt\oplus\RR C\oplus \RR C'$, which is moreover conformal with respect to $\Omega'$. But by Proposition \ref{maxdim} this is impossible.
\end{proof}

From now on, to simplify notations, denote by $\Omega$ and by $\theta$ the forms $\Omega^C$ and $\theta^C$ obtained in Lemma \ref{average2}. Let $\pi_{i\Lt}:\Lt^c\rightarrow i\Lt$ and $\pi_{\Lt}:\Lt^c\rightarrow \Lt$ be the natural projections, and consider the vector field $B':=\pi_{i\Lt}(C)\in i\Lt$. Since $\Lt\subset\ker \theta$, we have $\theta(B')=\theta(C)=\la$, so let $B=-\frac{1}{\la}B'\in i\Lt$.  Moreover, since $J\pi_{i\Lt}=\pi_\Lt J$, we also have $JB=-\frac{1}{\la}\pi_\Lt(JC)\in\Lt$. Since $B$ is a difference of vector fields preserving $\Omega$, it also preserves $\Omega$ and so does $JB$, being in $\Lt$. 

Consider, on the minimal cover $\hat M$, the \Ka\ form $\Omega_0=e^{-\phi}\Omega$ with corresponding metric $g_0$, where $d\phi=\theta$. We have: 
\begin{equation*}
\LL_B\Omega_0=-\theta(B)\Omega_0=\Omega_0 \text{  and } \LL_{JB}\Omega_0=0.
\end{equation*} 
Let $\eta_0:=\iota_B\Omega_0$ and $f_0:=\norm B^2_{g_0}=\eta_0(JB)$. It follows, by the above: 
\[d\eta_0=\LL_B\Omega_0=\Omega_0.\]
Since $JB$ commutes with $B$ and preserves $\Omega_0$, it also preserves $\eta_0$. Hence we have:
\begin{align*}
0=\LL_{JB}\eta_0=d\iota_{JB}\eta_0+\iota_{JB}d\eta_0=df_0+J\eta_0
\end{align*}
implying:
\begin{equation}\label{Jeta}
\eta_0=Jdf_0 \text{ and } \Omega_0=d\eta_0=dd^cf_0.
\end{equation}
 
 Now, since both $B$ and $JB$ are $\Gamma$-invariant, we have, for any $\gamma\in\Gamma$:
 \begin{equation*}
 \gamma^*f_0=(\gamma^*\Omega_0)(B,JB)=e^{-\rho(\gamma)}f_0
 \end{equation*}
 hence $f_0$ is a $\Gamma$-equivariant function, which is moreover strictly positive. Therefore, 
 \begin{equation}
 \Omega':=\frac{1}{f_0}dd^cf_0\end{equation} 
 is a LCK form on $M$, in the same conformal class as $\Omega$, with corresponding Lee form $\theta'=-d\ln f_0$. 
 
 Finally, we have, by \eqref{Jeta}:
 \begin{equation*}
 \iota_{JB}\Omega'=\frac{1}{f_0}J\eta_0=-\frac{1}{f_0}df_0=\theta'
 \end{equation*}
 which implies that $-B$ is the fundamental vector field $(\theta')^\#$. In particular, since $B$ is both holomorphic and an infinitesimal automorphism of $\Omega'$, it is also Killing, so: 
 \begin{equation}
 \nabla \theta'=d\theta'=0
 \end{equation}
 implying that $(\Omega',\theta')$ is the Vaisman structure that we have been looking for. This ends the proof of the theorem.
 \end{proof}
 
 Now, as a consequence of the main result, of \cite[Proposition~5.4]{mmp} and of \cite[Theorem~9.1]{d} we have:
 
 \begin{corollary}\label{cor}
 Let $(M,J,[\Omega])$ be a toric LCK manifold, strict or not. If $(M,J)$ is \Ka ian, then $M$ is simply connected, and in particular $b_1(M)$, the first Betti number of $M$, is $0$. If not, then $b_1(M)=1$.
 \end{corollary}
 
 \section{Final Remarks}

\begin{remark}
In the last part of the proof, once we had the holomorphic symplectic infinitesimal automorphisms $B$ and $JB$, we could have applied directly a theorem of Kamishima and Ornea in \cite{ko}. Indeed, their result states that presence of such vector fields, generating a what they call holomorphic conformal flow, assures the existence, in the given conformal class, of a Vaisman metric. However, we chose to construct the Vaisman structure explicitly, hence also giving an alternative proof of that theorem. Remark that to switch from conformal to symplectic automorphisms of $\Omega$, one only needs to consider the corresponding form $\Omega_G$ to a Gauduchon metric in the conformal class, for then we have, cf. \cite[Proposition~3.8(i)]{mmp}, $\aut(M,J,\Omega_G)=\aut(M,J,[\Omega])$. 
\end{remark}

\begin{remark}
Lemma \ref{rankGamma} can also be seen, a posteriori, as a consequence of a result of \cite[Proposition~5.4]{mmp}, where it is shown that a compact toric Vaisman manifold has first Betti number $1$. 
\end{remark}

\begin{remark}
In the symplectic setting it follows, by using the Delzant classification, that every compact toric symplectic manifold is in fact a toric \Ka\ manifold, i.e. there exists an integrable complex structure, preserved by the torus action, compatible with the given symplectic form. It turns out that this does not happen in the LCS setting, i.e. a toric LCS manifold is not necessarily toric LCK. We illustrate this by the following simple example.
\end{remark}

\begin{example}

Consider on the compact manifold $M=(S^1)^4$ the action of $T=\TT^2$ given by: 
\[(e^{it_1},e^{it_2}).(e^{i\theta_1},e^{i\theta_2},e^{i\theta_3},e^{i\theta_4})=(e^{i(t_1+\theta_1)},e^{i(t_2+\theta_2)},e^{i\theta_3},e^{i\theta_4})\]
where $\theta_1,\ldots, \theta_4$ and $t_1,t_2$ are the polar coordinates on $M$, respectively $T$, given by the exponential map $\exp:\RR\rightarrow S^1, \theta\mapsto e^{i\theta}$. Let $\nu$ be the volume form on $S^1$ such that $\exp^*\nu=d\theta$, let $p_j:M\rightarrow S^1$ be the canonical projection on the $j$-th factor of $M$ and let $\nu_j:=p_j^*\nu$, where $j\in\{1,2,3,4\}$. Define the $T$-invariant 1-forms on $M$:
\[\theta:=\nu_4 \text{ and } \eta:=\sin\theta_3\nu_1+\cos\theta_3\nu_2.\]
Then we have:
\[d\eta=\cos\theta_3\nu_3\wedge\nu_1-\sin\theta_3\nu_3\wedge\nu_2   \text{ and } d\eta\wedge\eta=\nu_3\wedge\nu_1\wedge\nu_2\]
implying that $\eta$ induces a contact form on $p_1^*S^1\times p_2^*S^1\times p_3^*S^1$, hence $\Omega=d^\theta\eta\in\Omega^2(M)$ is a $T$-invariant LCS form on $M$. Moreover, clearly $\Lie(T)\subset\ker\theta$, so $\Lie(T)\subset \aut'(M,[\Omega])$. Thus, by Remark \ref{exacthamiltonian}, the $T$-invriant form $-\eta$ gives a moment map for the $T$-action, hence $(M,[\Omega])$ is a toric LCS manifold. On the other hand, $b_1(M)=4\notin\{0,1\}$, so by Corollary \ref{cor} $M$ cannot admit a toric LCK structure (strict or not).
\end{example}

Theorem \ref{TheoremAA} together with the papers \cite{p}, \cite{mmp} and \cite{l}  lead to a classification of toric LCK manifolds, at least as toric complex manifolds. Indeed, now we know that we have a toric Vaisman metric on any toric LCK manifold. Next, since the universal cover of the Vaisman manifold is the \Ka\ cone over a Sasaki manifold, \cite{p} shows that the corresponding Sasaki manifold is also toric in a natural way. Moreover, Lemma \ref{rankGamma} or also \cite[Proposition~5.4]{mmp} imply that the Sasaki manifold is compact. Finally, any toric Sasaki manifold is in particular a toric contact manifold, and the last ones, when compact, were given a classification in \cite{l}.

\end{document}